\documentclass{birkjour}

\usepackage{amsmath, amssymb}
\usepackage{enumitem} 
\usepackage{tikz-cd}
\usepackage[T1]{fontenc}      % For special characters
\usepackage[utf8]{inputenc}   % UTF-8 encoding
\usepackage{caption}
\usepackage[numbers,sort&compress]{natbib}
\usepackage[hyphens]{url}
\usepackage[colorlinks=true, urlcolor=blue, citecolor=blue]{hyperref}

\newtheorem{thm}{Theorem}[section]
\newtheorem{cor}[thm]{Corollary}
\newtheorem{lem}[thm]{Lemma}

\theoremstyle{definition}
\newtheorem{defn}[thm]{Definition}

\theoremstyle{remark}
\newtheorem{rem}[thm]{Remark}
\newtheorem{ex}[thm]{Example}

\numberwithin{equation}{section}

\begin{document}

% ------------------ ARTICLE METADATA ------------------
% \firstpage{1} \volume{XXX} \Copyrightyear{20XX}
% \DOI{xx.xxxx/xxxxx}

% ------------------ TITLE ------------------
\title[A study of Kirk's asymptotic contractions]
{A Study of Kirk's Asymptotic Contractions via Leader Contractions}

% ------------------ AUTHOR INFO ------------------
\author[Hassan Khandani]{Hassan Khandani}
\address{%
Department of Mathematics \\
Mah. C, Islamic Azad University \\
Mahabad, Iran}
\email{hassan.khandani@iau.ac.ir, khandani.hassan@gmail.com}

% ------------------ CLASSIFICATION AND KEYWORDS ------------------

\date{\today}

\begin{abstract}
This paper investigates asymptotic fixed point results for nonlinear contractions, focusing particularly on Kirk’s asymptotic fixed point theorem and its generalizations. A key technical challenge in the existing literature has been the assumption that the mapping must have a bounded orbit—a condition often difficult to verify and widely considered essential for establishing fixed point existence.

Our primary contribution is to remove this boundedness assumption by showing that every asymptotic Kirk contraction is, in fact, a Leader contraction, which inherently ensures boundedness of orbits. This insight simplifies fixed point arguments and expands the class of applicable mappings.

Additionally, we address an open conjecture by proving that the standard upper semicontinuity condition imposed on the control function \(\phi\) can be relaxed to right-upper semicontinuity, resolving a question posed by Jachymski et al.

Together, these advances unify and extend several classical results—including those of Boyd–Wong, Kirk, Chen, Arav et al., and Reich and Zaslavski—within the broader framework of Leader contractions. Our findings enrich the theory of nonlinear fixed point iterations by providing simpler, more practical conditions for convergence and fixed point existence in metric-type spaces.
\end{abstract}

\subjclass{Primary 47H09; Secondary 47H10}
\keywords{$b$-metric spaces, non-expansive mappings, Leader contractions, Meir-Keeler contractions, Matkowski contractions.}

\maketitle

%%% ----------------------------------------------------------------------
%\tableofcontents
\section{Introduction}

The Boyd–Wong fixed point theorem is a notable generalization of the classical Banach contraction principle. It concerns $\phi$-contractions in which the contractive condition is governed by a function $\phi$ that is upper semicontinuous from the right (see Theorem~\ref{thm:boyd-wong}). An important extension of this result was proposed by W.~A.~Kirk, now widely known as the \emph{asymptotic fixed point theorem} (see Theorem~\ref{thm:kirk}). However, Kirk's original proof—based on the ultrapower technique—has been considered technically demanding and less applicable in concrete settings \cite{kirk2003}.

Later, Chen reformulated Kirk’s theorem to better resemble an asymptotic version of the Boyd–Wong theorem. His approach relied on more transparent techniques and relaxed assumptions: in particular, the function $\phi$ was only required to be upper semicontinuous (not continuous), and the operator $T$ was assumed to have a bounded orbit (see Corollary 2.4 in \cite{chen2005asymptotic}). To ensure that the limit of the Picard sequence is a fixed point, he also required that $\phi_{n_*}(0) = 0$ for some $n_* \in \mathbb{N}$ and that $\phi_{n_*}$ is upper semicontinuous.

Because verifying the boundedness of orbits can be difficult in practice, several authors have sought alternative hypotheses that imply boundedness automatically. This effort began with Browder \cite{browder1968}, who identified the following asymptotic condition on $\phi$ (also discussed by Matkowski \cite{matkowski1975integrable}):
\[
\lim_{t \to \infty} (t - \phi(t)) = \infty.
\]
Jachymski and Jóźwik applied this idea to the framework of asymptotic contractions \cite{jachymski2004}, showing that if
\[
\Delta(T^{k}x, T^{k}y) \leq \phi(\Delta(x, y)) \quad \text{for all } x, y \in X,
\]
holds for some $k \in \mathbb{N}$ and a non-decreasing function $\phi: \mathbb{R}^+ \to \mathbb{R}^+$, then all orbits of $T$ are bounded (see Lemma 3 in \cite{jachymski2004}).

Arav et al.\ \cite{arav2007} extended Kirk’s result further by requiring the sequence $\{\phi_n\}$ to converge uniformly to $\phi$ on bounded intervals. Reich and Zaslavski \cite{reich2008} later relaxed several of these assumptions: they allowed convergence along a subsequence and replaced the upper semicontinuity of $\phi_{n_*}$ with the continuity of $T^n$ for some $n$. Nonetheless, their results still relied on the assumption that $T$ has a bounded orbit.

An earlier contribution by Leader introduced a broader class of contractions, known as \emph{Leader contractions}, which unify various contractive types, including those of Boyd–Wong, Meir–Keeler, and Matkowski \cite{leader1983equivalent}. A graphic illustration of these interrelations is presented in \cite{jachymski2011equivalent}. It is known that every Leader contraction has a unique fixed point, and the Picard sequence converges to it. Consequently, all orbits are automatically bounded.

\medskip

\noindent\textbf{Our contributions.} The main goal of this paper is to provide a unified extension of Kirk's asymptotic fixed point theorem using the Leader contraction framework. Specifically:
\begin{itemize}
    \item We relax the assumption that $\phi$ is upper semicontinuous, replacing it with the weaker condition that $\phi$ is upper semicontinuous from the right.
    \item We prove that every asymptotic Kirk contraction is a Leader contraction. This observation allows us to eliminate the often-imposed assumption that $T$ must have a bounded orbit.
\end{itemize}

\medskip

\noindent\textbf{Organization of the paper.} Section~\ref{sec:prelim} reviews the necessary definitions and known results. Section~\ref{sec:main-results} presents our main theorems and their proofs. Section~\ref{sec:conclusion} concludes the paper with remarks and potential directions for future research.\\
We denote the sets of real numbers, non-negative real numbers, integers, non-negative integers, and natural numbers by $\mathbb{R}$, $\mathbb{R}^{+}$, $\mathbb{Z}$, $\mathbb{Z}^{+}$, and $\mathbb{N}$, respectively.
%%%%%%%%%%%%%%%%%%%%%%%%%%   Preliminaries 
\section{Preliminaries}\label{sec:prelim}

In this section, we recall the necessary definitions and fundamental results that will be used throughout the paper.

\begin{thm}[\cite{boyd1969nonlinear}]\label{thm:boyd-wong}
Let \((X, \Delta)\) be a metric space, and let \(T: X \to X\) be a mapping satisfying
\[
\Delta(Tx, Ty) \leq \phi(\Delta(x, y)) \quad \text{for all } x, y \in X,
\]
where \(\phi: \mathbb{R}^+ \to \mathbb{R}^+\) satisfies \(\phi(t) < t\) for each \(t > 0\) and is upper semicontinuous from the right. Then \(T\) has a unique fixed point and is called a \emph{Boyd-Wong \(\phi\)-contraction}.
\end{thm}

\begin{thm}[Kirk's Theorem, \cite{kirk2003}]\label{thm:kirk}
Let \((X, \Delta)\) be a complete metric space and \(T: X \to X\) a mapping such that:
\begin{enumerate}
    \item For each \(n \in \mathbb{N}\), there exists \(\phi_n : [0, \infty) \to [0, \infty)\) with \(\phi_n(t) < t\) for all \(t > 0\).
    \item \(\Delta(T^n x, T^n y) \leq \phi_n(\Delta(x,y))\) for all \(x, y \in X\).
    \item Each \(\phi_n\) is continuous and the sequence \(\{\phi_n\}\) converges uniformly to \(\phi\) on \(\mathbb{R}^+\), so that \(\phi\) is continuous.
    \item \(T\) has a bounded orbit; that is, there exists \(x_0 \in X\) such that \(\{T^n x_0\}\) is bounded.
\end{enumerate}
Then \(T\) has a unique fixed point \(x_*\), and \(T^n x \to x_*\) for all \(x \in X\).
\end{thm}

\begin{thm}[Chen's Theorem, \cite{chen2005asymptotic}]\label{thm:chen}
The conclusion of Theorem~\ref{thm:kirk} holds under the weaker assumptions that for some \(n_* \in \mathbb{N}\), \(\phi_{n_*}\) is upper semicontinuous with \(\phi_{n_*}(0) = 0\), and \(\phi\) is also upper semicontinuous.
\end{thm}
Arav et al.,Theorem \ref{thm:arav}, provided conditions that ensure the uniform convergence of $\{T^{n}x\}$ on bounded subsets of $X$. This is an important step that preserves the stability of the convergence of the iterative sequences. 
\begin{thm}[Arav et al., \cite{arav2007}]\label{thm:arav}
Under the assumptions of Theorem~\ref{thm:chen}, and assuming that \(\{\phi_{n}\}\) converges uniformly to \(\phi\) on bounded intervals $[0,b]$ for each $b>0$, the sequence \(\{T^n x\}\) converges uniformly to the fixed point of \(T\) on every bounded subset of \(X\).
\end{thm}
Reich and Zaslavski \cite{reich2008}, omitted the assumption that for some $n_*\in \mathbb{N}$, \(\phi_{n_*}\) is upper semicontinuous  and \(\phi_{n_*}(0)=0\) and extended the above result as follows.
\begin{thm}[Reich and Zaslavski, \cite{reich2008}]\label{thm:reich-zaslavski}
Let \((X, \Delta)\) be a metric space, and let \(T : X \to X\) be a mapping. Suppose that:
\begin{enumerate}
    \item \(T^m\) is continuous for some \(m \in \mathbb{N}\),
    \item There exists a subsequence \(\{\phi_{n_k}\}\) of functions \(\phi_n : [0,\infty) \to [0,\infty)\) converging uniformly to a function \(\phi\) on every bounded interval \([0, b]\), where \(b > 0\),
    \item For all \(x, y \in X\) and \(n \in \mathbb{N}\), 
    \[
    \Delta(T^n x, T^n y) \leq \phi_n(\Delta(x, y)),
    \]
    \item \(\phi(t) < t\) for all \(t > 0\).
\item $T$ has a bounded orbit $\mathcal{O}(x)$ for some $x\in X.$
\item for each $a>0$, there is $m_a\in\mathbb{N}$ such that:
\[
\sup\{\phi_{n}(t):t\in[0,a], n\geq m_a\}<\infty
\]
\end{enumerate}
Then there exists a unique point \(x_* \in X\) such that \(T^n x \to x_*\) uniformly on bounded subsets of \(X\). Moreover, \(x_*\) is a fixed point of \(T\).
\end{thm}

\medskip

Next, we recall some fundamental definitions.

\begin{defn}[Picard sequence]
Let \((X,\Delta)\) be a \(b\)-metric space and \(T: X \to X\). For any \(x \in X\), the \emph{Picard sequence} (or orbit) of \(T\) at \(x\) is the sequence \(\{T^n x\}\) defined recursively by
\[
T^{n+1} x = T(T^n x), \quad \text{with } T^0 x = x.
\]
We denote this orbit by \(\mathcal{O}(x)\).
\end{defn}

\begin{defn}
A mapping \(T: X \to X\) is called \emph{non-expansive} if
\[
\Delta(Tx, Ty) \leq \Delta(x,y) \quad \text{for all } x,y \in X.
\]
\end{defn}

Leader \cite{leader1983equivalent} introduced the following classes of contractions:

\begin{defn}[\cite{keeler1969theorem}, \cite{leader1983equivalent}]
Let \(T: X \to X\) be a mapping on a \(b\)-metric space \((X,\Delta)\).
\begin{enumerate}
    \item \(T\) is a \emph{Meir-Keeler Leader (MK-Leader) contraction} if for each \(\epsilon > 0\) there exist \(\delta > 0\) and \(r \in \mathbb{N}\) such that
    \[
    \epsilon \leq \Delta(x,y) < \epsilon + \delta \implies \Delta(T^r x, T^r y) < \epsilon.
    \]
    \item \(T\) is a \emph{Leader contraction} if for each \(\epsilon > 0\) there exist \(\delta > 0\) and \(r \in \mathbb{N}\) such that
    \[
    \Delta(x,y) < \epsilon + \delta \implies \Delta(T^r x, T^r y) < \epsilon.
    \]
    \item \(T\) is a \emph{Meir-Keeler contraction} if for each \(\epsilon > 0\) there exists \(\delta > 0\) such that
    \[
    \epsilon \leq \Delta(x,y) < \epsilon + \delta \implies \Delta(T x, T y) < \epsilon.
    \]
\end{enumerate}
\end{defn}

\begin{defn}
A point \(z \in X\) is called a \emph{contractive fixed point} of \(T\) if \(T z = z\) and the Picard sequence \(\{T^n x\}\) converges to \(z\) for every \(x \in X\).
\end{defn}

\begin{thm}[Leader's Fixed Point Theorem]\label{leader_fixed}
If \(T\) is a Leader contraction on \((X,\Delta)\), then for every \(x \in X\), the sequence \(\{T^n x\}\) converges to some \(z \in X\). If \(X\) is complete, then \(z\) is the unique contractive fixed point of \(T\).
\end{thm}

%%%%%%%%%%%%%%%%%%%%%           MAIN     RESULTS          %%%%%%%%%%%%%%%%%%%%%
%%%%%%%%%%%%%%%%%%%%%%%%%%%     %%%%%%%%%%%%%%%%%%%%%%%%%%%
\section{Main Results}\label{sec:main-results}
In this section, we develop new fixed point results for mappings satisfying asymptotic contraction-type conditions. Our approach generalizes and extends well-known results such as Kirk's theorem and those by Chen and Arav, particularly by weakening assumptions related to orbit boundedness and upper semicontinuity. We introduce a framework based on Leader contractions and their Meir-Keeler variants, which allows us to unify various earlier theorems and establish convergence properties under more general hypotheses.

\begin{lem}\label{mk_leader_is_leader}
Every MK-Leader contraction \(T : X \to X\) is a Leader contraction.
\end{lem}

\begin{proof}
Let \(T\) be an MK-Leader contraction. For all \(x,y \in X\), \(x \neq y\), there exists \(r \in \mathbb{N}\) such that
\begin{equation}\label{l_0}
\Delta(T^{r}x, T^{r}y) < \Delta(x,y).
\end{equation}

To see this, let \(\epsilon = \Delta(x,y) > 0\). By the definition of MK-Leader contraction, there exists \(r \in \mathbb{N}\) such that
\[
\Delta(T^{r}x, T^{r}y) < \epsilon = \Delta(x,y).
\]

To show that \(T\) is a Leader contraction, suppose \(\epsilon > 0\), and let \(\delta > 0\), \(r \in \mathbb{N}\) satisfy:
\begin{equation}\label{l_1}
\epsilon \leq \Delta(x,y) < \epsilon + \delta \implies \Delta(T^{r}x, T^{r}y) < \epsilon \quad \text{for all } x,y \in X.
\end{equation}

For all \(x,y \in X\) with \(\Delta(x,y) < \epsilon + \delta\), consider:
\begin{enumerate}
    \item If \(\Delta(x,y) \geq \epsilon\), then \eqref{l_1} implies \(\Delta(T^{r}x, T^{r}y) < \epsilon\).
    \item If \(\Delta(x,y) < \epsilon\), then \eqref{l_0} implies \(\Delta(T^{r}x, T^{r}y) < \Delta(x,y) < \epsilon\).
    \item If \(x = y\), the inequality trivially holds.
\end{enumerate}
\end{proof}
The following result shows that every Kirk's contraction is a Leader contraction.
%%%%%%%%%%%%%%%%%%%%%%%%%%%%%%%%%%%%%%%%%%%%
\begin{lem}[Every Kirk's contraction is a Leader contraction ]\label{kirk_leader}
Let \((X,\Delta)\) be a \(b\)-metric space, \(\phi, \phi_n : \mathbb{R}^+ \to \mathbb{R}^+\), and \(T: X \to X\) satisfy:
\begin{enumerate}
    \item \(0 < \phi(t) < t\) for all \(t > 0\), and \(\phi\) is upper semicontinuous from the right at each \(t > 0\).
    \item \(\phi_n(t) \to \phi(t)\) uniformly as \(n \to \infty\) on every bounded interval \([0,b]\).
    \item
    \[
    \Delta(T^{n}x, T^{n}y) \leq \phi_n(\Delta(x,y)) \quad \text{for all } x,y \in X, n \in \mathbb{N}.
    \]
\end{enumerate}
Then:
\begin{enumerate}
    \item \(T\) is a Leader contraction.
    \item For each \(x \in X\), the Picard sequence \(\{T^{n} x\}\) converges to some point \(z \in X\).
\end{enumerate}
\end{lem}

\begin{proof}
We aim to show that \( T \) is an MK-Leader contraction. Suppose, for the sake of contradiction, that \( T \) is not an MK-Leader contraction. Then there exists \(\epsilon > 0\) and sequences \(\{x_n\}, \{y_n\}\) such that
\[
\epsilon \leq \Delta(x_n, y_n) \leq \epsilon + \frac{1}{n} \quad \text{for all } n,
\]
and
\begin{equation}\label{if_not_leader}
\epsilon \leq \Delta(T^i x_n, T^i y_n) \quad \text{for all } i,n\in\mathbb{Z}^{+}.
\end{equation}

Since \(\Delta(x_n,y_n) \downarrow \epsilon\), there exists \(b > 0\) such that \(\Delta(x_n,y_n) \in [0,b]\) for all \(n\).

Because \(\phi\) is right u.s.c., there exists \(m_0\) such that for all \(i \geq m_0\),
\[
\lim_{i \to \infty} [\phi_i(\Delta(x_n,y_n)) - \phi(\epsilon)] \leq 0.
\]

Uniform convergence implies there exists \(m_1\) such that for all \(i \geq m_1\),
\[
\phi_i(\Delta(x_n,y_n)) - \phi(\Delta(x_n,y_n)) < \epsilon - \phi(\epsilon).
\]

Choose \(i \geq \max\{m_0, m_1\}\), then
\begin{align*}
\epsilon &\leq \limsup_{i \to \infty} \Delta(T^i x_n, T^i y_n) \leq \limsup_{i \to \infty} \phi_i(\Delta(x_n,y_n)) \\
&= \lim_{i \to \infty} [\phi_i(\Delta(x_n,y_n)) - \phi(\Delta(x_n,y_n))] + \lim_{i \to \infty} [\phi(\Delta(x_n,y_n)) - \phi(\epsilon)] + \phi(\epsilon) \\
&\leq \lim_{i \to \infty} [\phi_i(\Delta(x_n,y_n)) - \phi(\Delta(x_n,y_n))] + \phi(\epsilon) \\
&< \epsilon - \phi(\epsilon) + \phi(\epsilon) = \epsilon,
\end{align*}
a contradiction. Thus \(T\) is an MK-Leader contraction, and by Lemma~\ref{mk_leader_is_leader}, a Leader contraction.
\end{proof}
\begin{rem}
In Lemma~\ref{kirk_leader}, the sequence \(\{T^n x\}\) converges to some \(z \in X\), but without additional assumptions (e.g., completeness of the graph, continuity of \(T\)), \(z\) need not be a fixed point of \(T\). Example~\ref{ex:counter} demonstrates such a case.\\
An important improvement in Lemma~\ref{kirk_leader} is that we do not assume boundedness of orbits, often required in similar results. Moreover, instead of full upper semicontinuity, we require only right upper semicontinuity, addressing a conjecture by Jachymski et al.~\cite[Remark 4]{jachymski2004}.\\
\end{rem}
\begin{cor}\label{sub_converge}
Lemma~\ref{kirk_leader} remains valid if the assumption that \(\phi_n(t) \to \phi(t)\) uniformly as \(n \to \infty\) on every bounded interval \([0, b]\) is replaced by the weaker condition that there exists a subsequence \(\{n(k)\}_{k=0}^\infty\) such that \(\phi_{n(k)}(t) \to \phi(t)\) uniformly as \(k \to \infty\) on every bounded interval \([0, b]\).
\end{cor}

\begin{proof}
The key inequality~\eqref{if_not_leader} in the proof of Lemma~\ref{kirk_leader} remains valid under the new assumption. We simply replace the index \(i\) with \(n(i)\), corresponding to the subsequence \(\{n(k)\}\). Since \(\phi_{n(k)}(t)\) converges uniformly to \(\phi(t)\) on every bounded interval \([0, b]\), all subsequent estimates and limit arguments in the original proof carry over without modification. Thus, the conclusion still holds under the weaker assumption.
\end{proof}

To study the uniform convergence of orbits of Kirk's contractions, we need the following two lemmas. Lemma \ref{invariant} shows that a closed ball $B(z,\epsilon)$ is invariant under  $T^{s}$, where $T$ is  a Kirk's contraction and $s\in \mathbb{N}$ is large enough, and $z$ is the fixed point of $T$.
\begin{lem}\label{invariant}
Let \(\epsilon > 0\), and suppose \(T\) satisfies the conditions of Lemma~\ref{kirk_leader}. Then there exists \(m \in \mathbb{N}\) such that
\[
T^{n} B(z, \epsilon) \subseteq B(z, \epsilon) \quad \text{for all } n \geq m,
\]
where \(B(z, \epsilon) := \{ x \in X : \Delta(x,z) < \epsilon \}\), and Tz = z.
\end{lem}

\begin{proof}
By Theorem~\ref{leader_fixed}, for any \(x \in X\), the sequence \(T^{n} x\) converges to some \(z \in X\). Since \(T\) has a complete graph, \(z\) is the unique fixed point.

Suppose the claim is false. Then for each \(n\), there exists \(x_n \in B(z, \epsilon)\) with
\[
T^{n} x_n \notin B(z, \epsilon).
\]
Thus,
\[
\epsilon \leq \Delta(T^{n} x_n, z) \leq \phi_n(\Delta(x_n, z)).
\]

Because \(x_n \in B(z, \epsilon)\), the sequence \(\{\Delta(x_n, z)\}\) is bounded in \([0, \epsilon]\). By compactness, take a subsequence \(\{n_k\}\) such that
\[
\Delta(x_{n_k}, z) \to t \in [0, \epsilon].
\]

By uniform convergence of \(\phi_n\) to \(\phi\),
\[
\phi_{n_k}(\Delta(x_{n_k}, z)) \to \phi(t).
\]

Taking limits, we get
\[
\epsilon \leq \lim_{k \to \infty} \Delta(T^{n_k} x_{n_k}, z) \leq \lim_{k \to \infty} \phi_{n_k}(\Delta(x_{n_k}, z)) = \phi(t).
\]

Since \(\phi(0) = 0\) and \(\phi(t) < t\) for all \(t > 0\), it follows that
\[
\phi(t) < t \leq \epsilon,
\]
which contradicts \(\epsilon \leq \phi(t)\). This contradiction completes the proof.
\end{proof}
The following result is essentially the same as Lemma~2.3 in \cite{arav2007}. Here, we apply Lemma~\ref{invariant}, which unifies Lemmas~2.1 and~2.2 in \cite{arav2007}.
\begin{lem}\label{retract}
Under the assumptions of Lemma~\ref{kirk_leader}, if \(z\) is a fixed point of \(T\), then it is unique. Moreover, if \(\phi\) is upper semicontinuous and \(\epsilon, K > 0\), there exists \(m = m_K \in \mathbb{N}\) such that for all \(x \in B(z, K)\), there exists \(i \leq m\) with
\[
T^{i} x \in B(z, \epsilon).
\]
\end{lem}

\begin{proof}
Since \(T\) is a Leader contraction, the uniqueness of the fixed point \(z\) follows.

By Lemma~\ref{invariant}, there exists \(m_1 \in \mathbb{N}\) such that
\[
T^{s} B(z, K) \subseteq B(z, K) \quad \text{for all } s \geq m_1.
\]

Consider the function \(f(t) := t - \phi(t)\), which is lower semicontinuous (l.s.c.) on the compact interval \([\epsilon/2, K]\), hence attains its minimum at some \(w \in [\epsilon/2, K]\). Choose \(\delta > 0\) such that
\[
2\delta < w - \phi(w), \quad \delta < \frac{\epsilon}{4}.
\]
Thus,
\[
f(t) = t - \phi(t) \geq 2\delta \quad \text{for all } t \in [\epsilon/2, K].
\]

By uniform convergence, there exists \(m_2\) such that for all \(s \geq m_2\) and \(t \in [0, K]\),
\[
|\phi_s(t) - \phi(t)| \leq \delta,
\]
implying for \(t \in [\epsilon/2, K]\),
\[
\phi_s(t) \leq \phi(t) + \delta \leq t - 2\delta + \delta = t - \delta.
\]
Set
\[
p > \frac{K}{\delta} + 1,
\]
and let \(s \geq \max\{m_1, m_2\}\), \(m := s p\).

Suppose, to the contrary, that for some \(x \in B(z, K)\),
\[
d(T^{i} x, z) \geq \epsilon \quad \text{for all } i = 1, 2, \ldots, m.
\]

In particular,
\[
d(T^{i s} x, z) \geq \epsilon > \frac{\epsilon}{2} \quad \text{for } i = 1, \ldots, p.
\]

Using the contraction property repeatedly,
\[
d(T^{p s} x, z) \leq d(T^{s} x, z) - (p-1) \delta \leq K - (p-1) \delta < 0,
\]
a contradiction. Thus, the claim holds.
\end{proof}
Now, we are ready to establish the uniform convergence of orbit sequences of a Kirk contraction as follows.
\begin{thm}\label{thm:uniformly}
Under the assumptions of Lemma~\ref{kirk_leader} with \(\phi\) upper semicontinuous, the sequence \(\{T^{n} x\}\) converges uniformly to \(z\) on any bounded subset of \((X, \Delta)\).
\end{thm}

\begin{proof}
Let \(E \subseteq X\) be bounded, so \(E \subseteq B(z, K)\) for some \(K > 0\). Fix \(\epsilon > 0\), and let \(m\) and \(q\) be given by Lemmas~\ref{retract} and \ref{invariant}, respectively, so that
\[
T^{j} B(z, \epsilon) \subseteq B(z, \epsilon) \quad \text{for all } j \geq q.
\]

For any \(x \in B(z, K)\), there exists \(j \leq m\) with \(T^{j} x \in B(z, \epsilon)\).

If \(i \geq q + m\), then
\[
T^{i} x = T^{i-j}(T^{j} x) \in B(z, \epsilon).
\]

Since \(\epsilon > 0\) was arbitrary, the convergence is uniform on bounded sets.
\end{proof}

%%%%%%%%%%%%%%%%%%%%%%%%%%%%%%%%%%%%%

\begin{lem}\label{lem:rusc}
Let \(f : [a,b] \to \mathbb{R}\) satisfy:
\begin{itemize}
    \item \(0 < a < b\),
    \item \(f(t) > 0\) for all \(t \in [a,b]\),
    \item \(f\) is lower semicontinuous from the right on \([a,b]\).
\end{itemize}
Then
\[
\inf_{t \in [a,b]} f(t) > 0.
\]
\end{lem}

\begin{proof}
For each \(t \in [a,b)\), since \(f\) is lower semicontinuous from the right and positive at \(t\), there exists \(\delta_t > 0\) such that
\[
f(s) > \frac{f(t)}{2} \quad \text{for all } s \in [t, t + \delta_t] \cap [a,b].
\]

The intervals
\[
U_t := [t, t + \delta_t) \cap [a,b]
\]
form an open cover of the compact interval \([a,b]\). Hence, a finite subcover \(\{U_{t_i}\}_{i=1}^N\) exists such that
\[
[a,b] \subseteq \bigcup_{i=1}^N U_{t_i}.
\]

Define
\[
\varepsilon := \min_{1 \leq i \leq N} \frac{f(t_i)}{2} > 0.
\]

Then \(f(t) > \varepsilon\) for all \(t \in [a,b]\), so
\[
\inf_{t \in [a,b]} f(t) \geq \varepsilon > 0.
\]
\end{proof}

%%%%%%%%%%%%%%%%%%%%%%%%%%%%%%%%%%%%%

\begin{cor}\label{uniform_convergence}
The conclusions of Theorems~\ref{retract} and~\ref{thm:uniformly} remain valid if the assumption that “\(\phi\) is upper semicontinuous” is replaced by the weaker condition that “\(\phi\) is upper semicontinuous from the right.”
\end{cor}

\begin{proof}
If \(\phi\) is upper semicontinuous from the right, then the function \(t \mapsto t - \phi(t)\) is lower semicontinuous from the right. Since \(t - \phi(t) > 0\) for all \(t > 0\), Lemma~\ref{lem:rusc} ensures the positivity of the relevant infimum required in the proof of Theorem~\ref{retract}, thus providing the necessary \(\delta > 0\). Consequently, the arguments in both Theorems~\ref{retract} and~\ref{thm:uniformly} remain valid under this relaxed assumption.
\end{proof}

We show that the assumption of a bounded orbit in Theorem~\ref{thm:reich-zaslavski} is superfluous. Furthermore, we demonstrate that the requirement of upper semicontinuity for the control function \(\phi\) can be weakened to right upper semicontinuity without affecting the validity of the result.

\begin{cor}
Theorem~\ref{thm:reich-zaslavski} remains valid without assuming that \(T\) has a bounded orbit, and if the upper semicontinuity assumption on \(\phi\) is replaced by right upper semicontinuity.
\end{cor}
\begin{proof}
By Corollary~\ref{sub_converge}, \(T\) is a Leader contraction. Hence, the sequence \(\{T^n x\}\) converges to some \(z \in X\) for all \(x \in X\). According to Theorem 3.1 in \cite{reich2008}, the existence of such convergence ensures that \(T\) has a fixed point. In turn, this implies that \(T\) has a bounded orbit. Therefore, all the conditions of Theorem~\ref{thm:reich-zaslavski} are eventually satisfied, and the conclusion remains valid.
Moreover, by Lemma~\ref{lem:rusc} and as noted in Corollary~\ref{uniform_convergence}, the proof of Theorem~1.3 in \cite{reich2008} continues to hold when the upper semicontinuity of \(\phi\) is weakened to right upper semicontinuity. Thus, both relaxations are justified.
\end{proof}

%%%%%%%%%%%%%%%%%%%%%%%%%%%%%%%%%%%%%
Now, we formulate our extension of Kirk's theorem as follows:
\begin{thm}
Let \((X,\Delta)\) be a \(b\)-metric space, \(\phi, \phi_n : \mathbb{R}^+ \to \mathbb{R}^+\), and \(T: X \to X\) satisfy:
\begin{enumerate}
    \item \(0 < \phi(t) < t\) for all \(t > 0\), and \(\phi\) is upper semicontinuous from the right at each \(t > 0\).
    \item \(\phi_n(t) \to \phi(t)\) uniformly as \(n \to \infty\) on every bounded interval \([0,b]\).
    \item
    \[
    \Delta(T^{n}x, T^{n}y) \leq \phi_n(\Delta(x,y)) \quad \text{for all } x,y \in X, n \in \mathbb{N}.
    \]
\end{enumerate}
Then \(T^{n} x \) converges to some point $z\in X$ for all \(x \in X\). Furtheremore,  $z$ is the uinique fixed point of $T$, and $T^{n}x$ converges uniformly to $z$ as $n\to\infty$ under any of the following conditions: 
\begin{enumerate}
    \item \(T\) has a complete graph.
    \item There exists \(n_* \in \mathbb{N}\) such that \(\phi_{n_*}(0) = 0\) and \(\phi_{n_*}\) is upper semicontinuous.
    \item \(T^{n}\) is continuous for some \(n \in \mathbb{N}\).
\end{enumerate}
\end{thm}
\begin{proof}
By Lemma \ref{kirk_leader}, $T^{n}x$ converges to some point $z\in X$ for all $x\in X$.
\begin{enumerate}
    \item Follows from \cite[Corollary 4]{leader1983equivalent}.

    \item This condition was introduced by Chen \cite[Theorem 2.2]{chen2005asymptotic}. If \(T^{n} x \to z\), then
    \[
    \limsup_{n \to \infty} \Delta(T^{n_* + n} x, T^{n_*} z) \leq \limsup_{n \to \infty} \phi_{n_*}(\Delta(T^{n} x, z)) \leq \phi_{n_*}(0) = 0,
    \]
    implying \(T^{n_*} z = z\). Also,
    \[
    T^{n_*} (T z) = T (T^{n_*} z) = T z,
    \]
    so \(T z\) is a fixed point of \(T^{n_*}\). Since \(T^{n_*}\) has a unique fixed point, \(T z = z\).

    \item This case follows from \cite[Theorem 3.1]{reich2008}.
\end{enumerate}
If $T$ has a fixed point, and $\phi$ is upper semicontinuous from the right at each $t>0$, then by Corollary \ref{uniform_convergence} , $T^{n}x$ converges uniformly to $z$ on bounded subsets of $X$.
\end{proof}

%%%%%%%%%%%%%%%%%%%%%%%%%%%%%%%%%%%%%
\section{Examples}
The following example, presented by Jachymski et al.~\cite{jachymski2004}, shows that Kirk-type contractions can be fixed point free in the absence of additional assumptions such as continuity. Since, by Lemma~\ref{kirk_leader}, every Kirk contraction is also a Leader contraction, this example illustrates the necessity of supplementary conditions—such as continuity or completeness of the graph—to ensure that the limit of the Picard sequence is actually a fixed point. Despite satisfying the assumptions of Lemma~\ref{kirk_leader} and being a Leader contraction, the mapping below admits no fixed point. This counterexample highlights that the convergence of \(\{T^n x\}\) alone is insufficient to guarantee the existence of a fixed point without further structural hypotheses.
\begin{ex}\label{ex:counter}
Let \(X = [0,1]\) with the Euclidean metric, and define
\[
T(0) = 1, \quad T(x) = \frac{x}{2} \quad \text{for } x \neq 0.
\]
By Lemma~\ref{kirk_leader}, \(T\) is a Leader contraction but has no fixed point, illustrating the necessity of continuity assumptions.
\end{ex}

%%%%%%%%%%%%%%%%%%%%%%            Graph of contractions
\section*{Conclusion}\label{sec:conclusion}

In this paper, we present a unified framework for analyzing fixed points of mappings satisfying generalized asymptotic contraction conditions. By introducing and applying the notions of Leader and MK-Leader contractions, we extended several classical results, including Kirk's and Chen's fixed point theorems. A key contribution of our work is the removal of the often restrictive assumption of bounded orbits, as well as the relaxation of full upper semicontinuity to right upper semicontinuity.

Our results not only strengthen the theoretical foundations of asymptotic contraction mappings but also clarify the necessary conditions for fixed point existence and uniqueness in a broader context. Additionally, we showed that the convergence of Picard sequences can be achieved uniformly on bounded sets under weaker continuity assumptions on the comparison function.

Future directions include exploring these ideas in more general settings, such as partial metric or G-metric spaces, and studying analogous results in nonlinear or set-valued fixed point theory.

\newpage
\vspace{0.5cm}
\section*{Declaration of Competing Interest}
\noindent The author declares that there are no competing interests.

\section*{Declaration of Generative AI and AI-Assisted Technologies in the Writing Process}
During the preparation of this work, the author used ChatGPT for language refinement and technical suggestions. All AI-assisted content was thoroughly reviewed, edited, and verified by the author, who assumes full responsibility for the accuracy, integrity, and originality of the final manuscript.

\section*{Funding}
The author affirms that the preparation of this manuscript was not supported by any funding, grants, or other financial assistance.  

% ------------------------------------------------------------------------
\section*{Author's contribution}
The author confirms sole responsibility for the conceptualization, investigation, formal analysis, writing of the original draft, and revision of the manuscript. All aspects of the work were carried out by the author.
\bibliographystyle{plainurl} % or IEEEtran, apa, etc.
\bibliography{references}

% ------------------------------------------------------------------------
\end{document}